\documentclass[11pt]{article}

\usepackage{amsthm}
\usepackage{amsmath}
\usepackage{amssymb}
\usepackage{enumerate}
\usepackage{fullpage}
\usepackage{graphicx}
\usepackage{caption}
\usepackage{mathtools}

\usepackage{latexsym,epsfig} 
\usepackage[latin1]{inputenc}
\usepackage{graphicx}
\usepackage{amsmath}
\usepackage{amsfonts}
\usepackage{amssymb}
\usepackage{amsthm}
\usepackage{epstopdf}

\makeatother

\newcommand{\fk}[1]{\frac{3#1}{4}}
\newcommand{\bto}[2]{b^{#1 / #2}}
\newcommand{\xto}[3]{{#1}^{#2/#3}}
\newcommand{\loss}[1]{\text{exp}\left(\frac{2}{1-\bto{#1}{4}}+\frac{4\sigma}{\sqrt{\mu}(1-\bto{#1}{8})}\right)}
\newcommand{\epstar}{\epsilon^*}
\newcommand{\prob}{\mathbb{P}}
\newcommand{\epfive}{\frac{\epsilon}{6}}

\newcommand{\prtwo}{\mathbb{P}\big(H(Y_{m_2}) = i\big)}
\newcommand{\prone}{\mathbb{P}\Big(\tau_k \big(H(Y_{m_1}) \big) = i\Big)}

\newtheorem{theorem1}{Theorem}[section]
\newtheorem{lemma1}[theorem1]{Lemma}
\newtheorem{def1}[theorem1]{Definition}
\newtheorem{claim}{Claim}
\newtheorem{fact}[theorem1]{Fact}

\newtheorem{cor}[theorem1]{Corollary}
\newtheorem{obs}[theorem1]{Observation}
\begin{document}

\begin{center}
\uppercase{\bf On the Density of Happy Numbers}
\vskip 20pt
{\bf Justin Gilmer}\\
{ Department of Mathematics, Rutgers, 110 Frelinghuysen Road
Piscataway, NJ 08854, USA}\\
{\tt jmgilmer@math.rutgers.edu}\\ 
\vskip 10pt
\end{center}
\vskip 30pt

\centerline{\bf Abstract}

\noindent

The happy function $H: \mathbb{N} \rightarrow \mathbb{N}$ sends a positive integer to the sum of the squares of its digits. A number $x$ is said to be happy if the sequence $\{H^n(x)\}^\infty_{n=1}$ eventually reaches 1 (here $H^n(x)$ denotes the $n$'th iteration of $H$ on $x$). A basic open question regarding happy numbers is what bounds on the density can be proved. This paper uses probabilistic methods to reduce this problem to experimentally finding suitably large intervals containing a high (or low) density of happy numbers as a subset. Specifically, we show that $\bar{d} > .18577$ and $\underline{d} < .1138$, where $\bar{d}$ and $\underline{d}$ denote the upper and lower density of happy numbers respectively. We also prove that the asymptotic density does not exist for several generalizations of happy numbers.

\date{today}

\section{Introduction}

   It is well known that if you iterate the process of sending a positive integer to the sums of the squares of its digits, you eventually arrive at either $1$ or the cycle
   \[4 \rightarrow 16 \rightarrow 37 \rightarrow 58 \rightarrow 89 \rightarrow 145 \rightarrow 42 \rightarrow 20 \rightarrow 4. \] 
   
   If we change the map, instead sending an integer to the sum of the cubes of its digits, then there are 9 different possible cycles (see Section \ref{sec:cubing}). Many generalizations of these kinds of maps have been studied. For instance, \cite{ref2} considered the map which sends an integer $n$ to the sum of the $e$'th power of its base-$b$ digits. In this paper, we study a more general class of functions. 
   
   \begin{def1} \label{def:def1}
   Let $b > 1$ be an integer, and let $h$ be a sequence of $b$ non-negative integers such that $h(0) = 0$ and  $h(1) = 1$. Define $H: \mathbb{Z^+} \rightarrow \mathbb{Z^+}$ to be the following function: for $n \in \mathbb{Z}^+$, with base-$b$ representation $n = \sum\limits_{i=0}^{k} a_ib^{i}$, $H(n) := \sum\limits_{i=0}^{k} h(a_i)$. We say $H$ is the $b$-happy function with digit sequence $h$.
   \end{def1}
   
   As a special case, the $b$-happy function with digit sequence $\{0,1,2^e, \ldots, (b-1)^e\}$ is called the $(e,b)$-function.
   
   \begin{def1}
   Let $H$ be any $b$-happy function and let $C \subseteq \mathbb{N}$. We say $n \in \mathbb{N}$ is type-$C$ if there exists $k$ such that $H^k(n) \in C$.
   \end{def1}

   For example, for the $(2,10)$-function, happy numbers are type-$\{1\}$. Numbers which are not happy are type-$\{4,16,37,58,89,145,42,20\}$. 
   
    Fix a $b$-happy function $H$ and let $\alpha := \max\limits_{i=0,\dots, b-1}\big(H(i)\big)$. If $n$ is a $d$-digit integer in base-$b$, then $H(n) \leq \alpha d$.
   If $d^*$ is the smallest $d \in \mathbb{N}$ such that $\alpha d < b^{d-1}$, then for all $n$ with $d \geq d^*$ digits, $H(n) \leq \alpha d < b^{d-1} \leq n$. This implies the following
   
   \begin{fact}
    For all $ n \in \mathbb{N}$, there exists an integer $i$ such that $H^i(n) < b^{d^*}-1$.
   \end{fact}
      
    Moreover, to find all possible cycles for a $b$-happy function, it suffices to perform a computer search on the trajectories of the integers in the interval $[0,b^{d^*} -1]$. 
    
    Richard Guy asks a number of questions regarding $(2,10)$-happy numbers and their generalizations, including the existence (or not) of arbitrarily long sequences of consecutive happy numbers and whether or not the asymptotic density exists [4, problem E34]. To date, there have been a number of papers in the literature addressing the former question (\cite{ref1},\cite{ref2},\cite{ref3}). This paper addresses the latter question. Informally, our main result says that if the asymptotic density exists, then the density function must quickly approach this limit. 
    
    \begin{theorem1} \label{mainthm}
     Fix a $b$-happy function $H$. Let $I$ be a sufficiently large interval and let $S \subseteq I$ be a set of type-$C$ integers. If $\frac{|S|}{|I|} = d$, then the upper density of type-$C$ integers is at least $d\left(1 - o(1)\right)$.
    \end{theorem1}
    
    Note as a corollary we can get an upper bound on the lower density by taking $C$ to be the union of all cycles except the one in which we are interested. In Sections 3 and 4 we will define explicitly what constitutes a sufficiently large interval and provide an expression for the $o(1)$ term. Using Theorem \ref{mainthm}, one can prove the asymptotic density of $(e,b)$-happy numbers (or more generally type-$\{C\}$ numbers) does not exist by finding two large intervals $I_1, I_2$ for which the density in $I_1$ is large and in $I_2$ is small. In the case of $(2,10)$-happy numbers, taking $I_1 = [10^{403},10^{404}-1]$ and $I_2 = [10^{2367},10^{2368}-1]$, we show that $\bar{d} \geq .185773(1-10^{-49})$ and $\underline{d} \leq  .11379(1+10^{-100})$ respectively.  
    
    We also show that the asymptotic density does not exist for $8$ of the cycles for the $(3,10)$-function (see Section 5). It should be noted that our methods only give one sided bounds. In an earlier version of this manuscript, we asked if $\bar{d} < 1$ for $(2,10)$-happy numbers. Recently, \cite{Moews} has announced a proof of this. Specifically, he proves that $.1962 < \bar{d} < .38$, and $0.002937 < \underline{d} < .1217$.

\section{Preliminaries}
  Throughout the paper we regard an interval $I = [a,b]$ as a set of integers  $\{n \in \mathbb{Z}^+ : a \leq n \leq b\}$ where, in general, $a,b \in \mathbb{R}$. We denote $|I|$ to be the cardinality of this set. We also denote the set $\{0,1,\dots,n\}$ by $[n]$. Throughout this section let $H$ denote an arbitrary $b$-happy function with digit sequence $h$.
  
  \begin{def1}
    Let $I$ be an interval and $Y$ the random variable uniformly distributed amongst the set of integers in $I$. Then we say the random variable $Y$ is induced by the interval $I$.
  \end{def1}

   \begin{def1}
   The type-$C$ density of an integer interval $I$ is defined to be the quantity 
   \[d_C(I):=\frac{|\{n \in I : n \text{ is  type-}C\}|}{|I|}.\]
   \end{def1}
   \begin{obs}
   If $Y$ is the random variable induced by an interval $I$, then  \[ d_C(I) =\mathbb{P}\big(H(Y) \text{ is type-C}\big).\] 
   \end{obs}
   Usually, we take $C$ to be one of the cycles arising from a $b$-happy function $H$. However, if we wish to upper bound the lower density of type-$C$ integers, then we study the density of type-$C'$ integers, where $C'$ is the union of all cycles except $C$.
   
   \subsection{The Random Variable $H(Y_m)$}
   Consider the random variable $Y_m$ induced by the interval $[0,b^m-1]$, i.e., $Y_m$ is a random $m$-digit number. If $X_i$ is the random variable corresponding to the coefficient of $b^i$ in the base-$b$ expansion of $Y_m$, then 
   \begin{equation} \label{eq:eq3}
      H(Y_m) = \sum\limits_{i=0}^{m-1}h(X_i).
   \end{equation}
   In this paper, we will be interested in the mean and variance of the $h(X_i)$ (i.e., the image of a random digit) which we refer to as the \textit{digit mean} ($\mu$) and \textit{digit variance} ($\sigma^2$) of $H$.
   The random variables $h(X_i)$ in (\ref{eq:eq3}) are all independent and identically distributed (i.i.d.), thus,
   
   \begin{equation} \label{eq:eq4}
     \textbf{E}[H(Y_m)] = \mu m \qquad \textbf{ Var}[H(Y_m)] = \sigma^2 m.
   \end{equation}
   
   The random variable $H(Y_m)$ is equivalent to rolling $m$ times a $b$-sided die with faces $0 ,1,h(2),\dots,h(b-1)$ and taking the sum. Since it is a sum of $m$ i.i.d. random variables, it approaches a normal distribution as $m$ gets large. Also, the distribution of $H(Y_m)$ is concentrated near the mean. This observation leads to the following key insight which underlies the proofs in this work: {\bf For a sufficiently large integer $m$, the density of type-$C$ integers among $m$ digit integers is approximately determined by the set of type-$C$ integers near $\mu m$.}

   \subsection{Computing Densities}
   In order to apply Theorem \ref{mainthm} it is necessary to compute the number of $m$-digit integers which are type-$C$ for $m$ large. In this section we discuss how this can be done efficiently (even for $m \geq 1000$).
   
   Let $P_{m,i} := \mathbb{P}\big(H(Y_m) = i\big)$. Then
   \[P_{m,i} = \frac{\Big| \{(a_1,a_2,\dots,a_m) : a_k \in h \text{ and } \sum\limits_{k=1}^{m} a_k = i \} \Big|}{b^{m}}. \]
   For fixed $m$, the sequence $\{P_{m,i}\}_{i=1}^{\infty}$ has generating function 
   \begin{equation} \label{eq:genf}
    f_m(x) = \sum\limits_{i=0}^{\infty}P_{m,i}x^i = \left( \frac{1 + x + x^{h(2)} + \dots + x^{h(b-1)}}{b} \right)^m.
   \end{equation}
   
   This implies the following recurrence relation with initial conditions $P_{0,0} = 1$, and $P_{0,i} = 0$ for $i \in \mathbb{Z}- \{0\} $.
   
   \begin{equation} \label{eq:eq2}
    P_{m,i} = \frac{P_{m-1,i} + P_{m-1,i-1} + P_{m-1,i-h(2)} + \dots + P_{m-1,i-h(b-1)}}{b}.
   \end{equation}
   
    To see this, write $f_m(x) = \left( \frac{1 + x + x^{h(2)} + \dots + x^{h(b-1)}}{b} \right)^{m-1}\left( \frac{1 + x + x^{h(2)} + \dots + x^{h(b-1)}}{b} \right)$ and consider the coefficient of $x^i$.

   If $\alpha = \max\limits_{i=0,\dots,b-1} \big(h(i)\big)$, then $H(Y_m) \subseteq [0,m\alpha]$. In particular, $P_{m,i} = 0$ if $i > m\alpha$. Using this fact combined with (\ref{eq:eq2}), we can implement the following simple algorithm for quickly calculating the type-$C$ density of the  interval $[0,b^m-1]$.
   
   \begin{enumerate}
   \item First, using the recurrence (\ref{eq:eq2}), calculate $P_{m,i}$ for $i = 0, \dots, m\alpha$.
   \item Using brute force, find the type-C integers in the interval $[0, m\alpha]$. 
   \item Output $\sum\limits_{\substack{i \in [0, m\alpha] \\ i \text{ type-C}}}P_{m,i}$. 
   \end{enumerate}
   
   Using this algorithm, calculating the density for large $m$ becomes computationally feasible. Figure 1 graphs the density of $(2,10)$-happy numbers $< 10^m$ for $m$ up to $8000$. 
   
   \begin{figure}[htb]
\centering
  \caption{Relative Density of $(2,10)$-Happy Numbers $< 10^m$}
  \includegraphics[scale=.7]{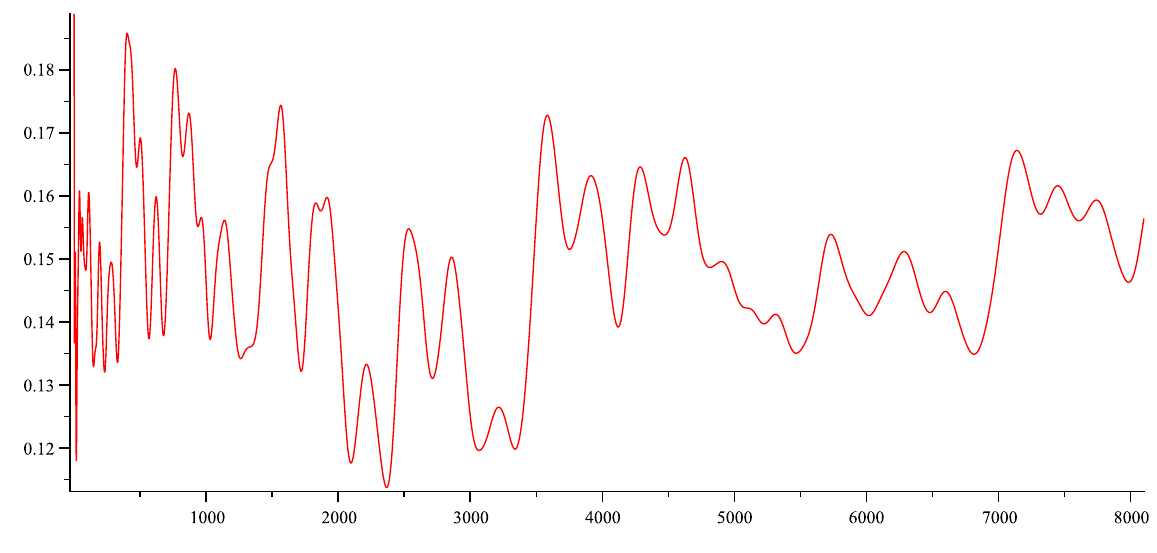} 
\end{figure}
   The peak near $10^{400}$ and valley near $10^{2350}$ will be used to imply the bounds obtained in this paper.

   \subsection{A Local Limit Law}
   The random variable $H(Y_m)$ approaches a normal distribution as $m$ becomes large. The following theorem\footnote{We quote a simpler version, with a minor typo corrected}, presented in [2, p. 593], gives a bound.
   \begin{theorem1} \label{thm:thm5}
    (Local limit law for sums). Let $X_1,\dots,X_n$ be i.i.d. integer-valued variables with probability generating function (PGF) $B(z)$, mean $\mu$, and variance $\sigma^2$, where it is assumed that the $X_i$ are supported on $\mathbb{Z}^+$. Assume that $B(z)$ is analytic in some disc that contains the unit disc in its interior and that $B(z)$ is aperiodic with $B(0) \neq 0$. Then the sum,
    \[S_n := X_1 + X_2 + \dots + X_n \]
    satisfies a local limit law of the Gaussian type: for $t$ in any finite interval, one has
    \[ \mathbb{P}(S_n = \lfloor \mu n + t \sigma \sqrt{n} \rfloor) = \frac{e^{-t^2/2}}{\sqrt{2\pi n}\sigma}\left(1 + O(n^{-1/2})\right). \]. 
    \end{theorem1}
    
    Here aperiodic means that the $\gcd\{j : b_j > 0, j > 0\} = 1$, where $B(z) = \sum\limits_{j = 0}^{\infty}b_j z^j$ (or more informally, the digit sequence for $H$ cannot all be divisible by some integer larger than $1$). In our case, the PGF of the $H(X_i)$ is the polynomial 
    \[ p(x) = \frac{x^{H(0)} + x^{H(1)} + \dots + x^{H(b-1)}}{b}. \]
    It is important in our definition of $b$-happy functions that we assume that $H(0) = 0$, and $H(1) = 1$. This guarantees that $p(x)$ is aperiodic and in particular that the above theorem applies for the sum $H(Y_m)$. As a consequence, for a fixed interval $[-T,T]$, if $i = \mu m + t \sigma \sqrt{m}$ for some $t \in [-T,T]$, then 
    \[ P_{m,i} = \frac{e^{-t^2/2}}{\sqrt{2\pi m}\sigma}\left(1 + O\left(m^{-1/2}\right)\right).\]
   
    The above error term, $O\left(m^{-1/2}\right)$, will prove to be a technical difficulty which will be discussed later.
    \subsection{Overview of the Proof}
    The following heuristic will provide the general motivation for the proofs. Recall that the random variable $Y_m$ is concentrated near its mean $\mu m$. 
    \begin{obs} \label{rem:rem1}
      Suppose $I$ is a large interval with type-$C$ density $d$. Consider the choices of $m$ such that the mean of $H(Y_m)$ is in the interval $I$, then for some choices of $m$ we likely have 
     \[ \mathbb{P}\big(H(Y_m) \text{ is type-}C\big) \geq d. \]
    \end{obs}
     The key idea to turn this heuristic into a proof is to average over all reasonable choices of $m$ in order to imply there is an $m$ with the desired property.
    
     We will use Theorem \ref{thm:thm5} to show that, for small $k$, $H(Y_m)$ and $H(Y_{m+k})$ have essentially the same distribution only shifted by a factor of $\mu k$. Thus, as $k$  varies, the distributions $H(Y_{m+k})$ should uniformly cover the interval $I$. It is crucial here to use the fact that $H(Y_m)$ is locally normal, otherwise the proof will fail. For example, suppose all the happy numbers in $I$ are odd. In this case, if $H(Y_m)$ is not locally normal and instead is supported on the even numbers for all $m$, then every shift $H(Y_{m+k})$ will miss all of the happy numbers in $I$.
     
      Unfortunately, the fuzzy term in the local limit law prevents us from obtaining explicit bounds on the error (and any explicit bounds seem unsatisfactory for our purposes). Section 3 adds a necessary step, which is to construct an interval within $[b^{n-1},b^n]$ with high type-$C$ density for $n$ arbitrarily large. The trick is to consider intervals of the form $I_k :=[1^k 0^{n-k}, 1^k 0^{m} (b-1)^{n-k-m}]$ (here $1^k$ denotes $k$ consecutive 1's). This solves the issue where $H(Y_m)$ and $H(Y_{m+k})$ are not exact shifts of each other, as the distributions induced by the $I_k$ \emph{are} exact shifts under the image of $H$. These distributions will uniformly cover the base interval $I$ with much better provable bounds. The main result is presented in Section 4, the proof uses the local limit law with the result from Section 3.

\section{Constructing Intervals}
  Throughout this section, if $Y$ is a random variable and $k$ is an integer, let $\tau_k(Y)$ denote the random variable $Y+k$.
  \begin{def1}
     We say an integer interval $I$ is $n$-strict if $I \subseteq [b^{n-1},b^n-1]$ and $ |I| = \bto{3n}{4}$.
  \end{def1}
  
  The primary goal of this section is to construct $n$-strict intervals of high type-C density for arbitrarily large $n$.
  
  Our choice of the definition of $n$-strict is only for the purpose of simplifying calculations, there is nothing special about the value $\frac{3}{4}$. In fact, any ratio $> \frac{1}{2}$ would work. Note if $4$ does not divide $n$, then no $n$-strict intervals exist.
  
 For the entirety of this section we will make the following assumptions:
 
 \begin{itemize} \label{eq:Hcond}
  \item $H$ is a $b$-happy function with digit mean $\mu$ and digit variance $\sigma^2$.
  \item We wish to lower bound the upper density of type-C integers for some $C \subset \mathbb{N}$.
  \item We have found, by computer search, an appropriate starting interval $I_1$, which is $n_1$-strict and has suitably large type-C density $d_C(I_1)$.
 \end{itemize}
 
  The results in this section apply only if this $n_1$ is sufficiently large, so we state here exactly how large $n_1$ must be so one knows where to look for the interval $I_1$. In particular, we say an integer $n$ satisfies the bounds (\textbf{B}) if
  
   \begin{description}
    \item[B1:] $4\left(1+3\mu + \sqrt{2}\sigma\bto{5n}{8}\right) \leq b^{n-1}$,
    \item[B2:] $\sqrt{3\mu b}\sigma \leq \bto{3n}{8}$,
    \item[B3:] $4\mu\left(3\mu + 1 + \bto{3n}{4} + 2\sigma\xto{\mu}{-1}{2}\bto{5n}{8}\right) \leq b^{n-1}$.
   \end{description}
  
  Generally, $n$ need not be too large to satisfy these bounds. For example, if $H$ is a $(2,10)$-happy function, assuming $n>13$ is enough to guarantee that it satisfies bounds (\textbf{B}). This is well within the scope of the average computer as it is possible to compute the density of type-$C$ integers in $[0,b^n-1]$ for $n$ up to (and beyond) $1000$ using the algorithm in Section 2. These bounds are necessary in the proof of Theorem \ref{thm:thm2}.
    
  Our first goal is to use an arbitrary $n$-strict interval $I$ to construct a second interval, $I_2$, which is $n_2$-strict for some $n_2$ much larger than $n$ and contains a similar density of type-C integers as $I$. The next lemma will be a helpful tool.
   
 \begin{lemma1} \label{lem:lem1}
   Let $I := [i_1,i_2]$, $J := [j_1,j_2]$ be integer intervals. Let $S \subseteq I$ and $Y$ be an integer-valued random variable whose support is in $J$. For $k \in \mathbb{Z}$, denote the random variable $Y+k$ as $\tau_k(Y)$. Then there exists an integer $k \in [i_1-j_2,i_2-j_1]$ such that $\prob\big(\tau_k(Y) \in S\big) \geq \frac{|S|}{|I|+|J|-1}$.
  \end{lemma1}
  \begin{proof}
   The idea of the proof is that by averaging over all appropriate $k$, the distributions of $\tau_k(Y)$ should uniformly cover $I$. More formally, let $k_1 :=  i_1 - j_2$, $k_2 := i_2 - j_1$, and let $K$ be the set of integers in the interval $[k_1,k_2]$. Note that $|K| = |I| + |J| - 1$.  Pick $k$ uniformly at random from $K$ and consider the random variable $Z := \mathbb{P}\big(\tau_k(Y) \in S\big)$.  Then
  \[\textbf{E}[Z] = \frac{1}{|K|} \sum_{k=k_1}^{k_2} \mathbb{P}\big(\tau_k(Y) \in S\big)\]
  \[ = \frac{1}{|I|+|J| - 1} \sum_{k=k_1}^{k_2} \sum_{i \in S} \mathbb{P}\big(\tau_k(Y) = i\big)\]
  \begin{equation} \label{eq:eq1}
   = \frac{1}{|I|+|J| - 1} \sum_{i \in S} \sum_{k=k_1}^{k_2} \mathbb{P}(\tau_k(Y) = i).
  \end{equation}
  Note that $\mathbb{P}\big(\tau_k(Y) = i\big) = \mathbb{P}\big(Y = i - k\big)$ and for $i \in S \subseteq I$ we have,
  \[J \subseteq [i-k_2,i-k_1].\]
   Thus, for all $i \in S$,
  \[ \sum\limits_{k=k_1}^{k_2} \mathbb{P}\big(\tau_k(Y) = i\big) = \mathbb{P}(Y \in [i-k_2,i-k_1]) = 1.\]
   Therefore,
  \[ (\ref{eq:eq1}) = \frac{|S|}{|I|+|J| - 1}.\]
  So there exists $k$ such that $\mathbb{P}\big(\tau_k(Y) \in S\big) \geq \textbf{E}[Z] = \frac{|S|}{|I|+|J| - 1}$.
  
  \end{proof}
  
  Using Lemma \ref{lem:lem1}, we will not lose much density assuming that $|I|$ is much larger than $|J|$. However, if $Y_m$ is induced by the interval $[0, b^m -1]$, then the random variable $H(Y_m)$ will be supported on a set $J$ that is much too large. As a result, it will be more useful to consider a smaller interval where the bulk of the distribution lies.
  
  \begin{lemma1} \label{lem:lem2}
   Let $Y$ be an integer-valued random variable with mean $\mu_Y$ and variance $\sigma_Y^2$, and let $\lambda > 0$. Let $S \subseteq [i_1,i_2] = I$ be a set of integers where $|S|/|I| = d$. Then there exists an integer $k \in [i_1 - (\mu_Y +\sigma_Y \lambda), i_2 - (\mu_Y - \sigma_Y \lambda)]$ such that 
   \[\mathbb{P}\big(\tau_k(Y) \in S\big) \geq \left(1-\frac{1}{\lambda^2}\right)\Big(\frac{d}{1 + \frac{2 \sigma_Y \lambda }{|I|}}\Big).\]
   \end{lemma1}
   \begin{proof}
    By Chebyshev's Inequality\footnote{We certainly could do better than Chebyshev's Inequality here. However, the bounds it gives will suit our purposes fine.} we have $\mathbb{P}(|Y - \mu_Y| < \sigma_Y\lambda) > 1 - \frac{1}{\lambda^2}$. Let $Y'$ be $Y$ conditioned on being in the interval $J := [ \mu_Y - \sigma_Y\lambda, \mu_Y + \sigma_Y\lambda]$. Note that \[|J| \leq 2 \sigma_Y\lambda +1.\] Then, by Lemma \ref{lem:lem1}, there exists $k \in [i_1 - (\mu_Y + \sigma_Y \lambda), i_2 - (\mu_Y - \sigma_Y \lambda)]$ such that $\prob\big(\tau_k (Y') \in S\big) \geq \frac{|S|}{|I| + |J| -1} = \frac{d}{1 + \frac{2 \sigma_Y \lambda}{|I|}}$.
    Therefore, we have \[\prob\big(\tau_k (Y) \in S\big) \geq \prob(Y \in J)\prob\big(\tau_k (Y') \in S\big) \geq \left(1-\frac{1}{\lambda^2}\right)\left(\frac{d}{1 + \frac{2 \sigma_Y \lambda}{|I|}}\right).\]
   \end{proof}
   
   It is possible to construct sets of intervals which, under the image of $H$, act as shifts of each other. For example, in base-10 (recall $H(0) = 0, H(1) = 1$) if the random variable $X_1$ is induced by $[1100,1199]$ and $X_2$ is induced by $[0,99]$, then $H(X_1) =  H(X_2)+ 2$.
   
   We will now further expand on the example above. Let $n \in \mathbb{N}$ be divisible by 4. Let $B_0 := [0, \bto{3n}{4}-1]$ and, for $k = 1, \dots, \frac{n}{4}$, consider the interval
   \[B_k := [b^{n-1} + b^{n-2} + \dots + b^{n-k},b^{n-1} + b^{n-2} + \dots + b^{n-k} + \bto{3n}{4}-1].\]
   Then the intervals $B_k$ will all be $n$-strict (with exception of $B_0$), and a random integer $x \in B_k$ will have the following base-$b$ expansion:
   \[ x = \underbrace{11\dots 1}_{ \text{ k digits}}\underbrace{00\dots 0}_{ \frac{n}{4}-k \text{ digits}}\underbrace{X_{i}X_{i-1}\dots X_1}_{ \frac{3n}{4} \text{ digits}}. \]
   That is, $x$ will have its first $k$ digits equal to $1$, the next $\frac{n}{4}-k$ digits equal to $0$, and the remaining $\frac{3n}{4}$ digits will be i.i.d. random variables $X_i$ taking values uniformly in the set $\{0, 1, \dots, b-1\}$.
   
   For $k = 0,\cdots, \frac{n}{4}$ let $Y_k$ be the random variable which is uniform in $B_k$. Note that
   \[H(Y_k) = H(Y_0) + k = \tau_k\left(H(Y_0)\right). \]
   
   Recall $H(Y_0)$ has mean $\frac{3n}{4}\mu$, and variance $\frac{3n}{4}\sigma^2$. Consider an interval $I = [i_1,i_2]$ containing a set of type-$C$ integers $S$, and let $\lambda > 0$. By Lemma \ref{lem:lem2}, there exists $k' \in \Big[i_1 - \Big(\frac{3n}{4}\mu + \sqrt{\frac{3n}{4}}\lambda\sigma\Big), i_2 - \Big(\frac{3n}{4}\mu - \sqrt{\frac{3n}{4}}\lambda\sigma\Big)\Big]$ such that
   \begin{equation} \label{eq:eq5}
    \mathbb{P}\Big(\tau_{k'}\big(H(Y_0)\big) \in S\Big) \geq \Big(1-\frac{1}{\lambda^2}\Big)\Big(\frac{d_C(I)}{1+\frac{\sqrt{3n}\lambda\sigma}{|I|}}\Big).
   \end{equation}
    Thus, if $I \subseteq \Big[1 + \frac{3}{4}n\mu + \lambda\sigma\sqrt{\frac{3}{4}n}, \frac{1}{4}n + \frac{3}{4}n\mu - \lambda\sigma\sqrt{\frac{3}{4}n}\Big]$, then $1 \leq k' \leq \frac{n}{4}$. Setting $k := k'$ produces  the interval $B_k$, which will be $n$-strict with 
    \[d_C(B_k) \geq \Big(1-\frac{1}{\lambda^2}\Big)\Big(\frac{d}{1+\frac{\sqrt{3n}\lambda\sigma}{|I|}}\Big).\]
     In fact, we have proven the following

   \begin{theorem1} \label{thm:thm1}
   Let $n \in \mathbb{N}$ be divisible by 4 and let $C \subset \mathbb{N}$. For $\lambda > 0$, define $J_{n,\lambda} := \Big[1 + \frac{3}{4}n\mu + \lambda\sigma\sqrt{\frac{3}{4}n}, \frac{1}{4}n + \frac{3}{4}n\mu - \lambda\sigma\sqrt{\frac{3}{4}n}\Big]$. Fix an interval $I \subseteq J_{n,\lambda}$. Then there exists an $n$-strict interval, $I_2$, such that $d_C(I_2) \geq d_C(I) \Big(1-\frac{1}{\lambda^2}\Big)\Big(\frac{1}{1 + \frac{ \sqrt{3n} \sigma \lambda}{|I|}}\Big)$.
   \end{theorem1} 
   
   The goal for the rest of the section is to use the previous theorem iteratively to construct a sequence of intervals $\{I_i\}_{i=1}^{\infty}$, each with high type-$C$ density, such that each $I_i$ is $n_i$-strict and the sequence $\{n_i\}_{i=1}^{\infty}$ grows quickly. One technical issue to worry about is that $d_C(I_{i+1}) < d_C(I_i)$ for all $i$. How much smaller $d_C(I_{i+1})$ is depends on how large we choose $\lambda_i$ to be in each step. We wish to choose $\lambda_i$ as large as possible, but choose $\lambda_i$ too large and two bad things can happen: First, $I_i$ will not be contained in $J_{n_{i+1},\lambda_i}$ for any choice of $n_{i+1}$. Second, $\frac{ \sqrt{3n} \sigma \lambda_i}{|I|}$ will not be small. We are helped by the fact that the sequence $\{n_i\}_{i=1}^{\infty}$ will grow super exponentially (in fact, $n_{i+1} = \Omega(b^{n_i})$). Choosing $\lambda_i = b^{n_i/8}$ in each step will work well; however, we will need the initial $n_1$ to be sufficiently large. The next theorem gives precise calculations. The proof follows from a number of routine calculations and estimations, some of which we have left for the appendix.
   
    \begin{theorem1} \label{thm:thm2}
    Suppose $I$ is $n$-strict, where $n$ satisfies bounds (\textbf{B}). Then there exists $n_2 \geq \frac{b^{n-1}}{\mu}$, and an $n_2$-strict interval $I_2$ such that 
    \[d_C(I_2) \geq d_C(I) \left(1-\bto{-n}{4}\right)\left(1- \frac{2\sigma}{\sqrt{\mu}}\bto{-n}{8}\right).\]
    \end{theorem1}
    \begin{proof}
    As before, let $J_{m,\lambda} := \left[1+\frac{3}{4}m\mu + \lambda \sigma \sqrt{\frac{3}{4}m}, \frac{1}{4}m + \frac{3}{4}m\mu - \lambda \sigma \sqrt{\frac{3}{4}m}\right]$. We assumed that $I$ is $n$-strict, so $|I| = \bto{3n}{4}$. Write $I$ as $[a,a+b^{3n/4}-1]$. Setting $\lambda := b^{n/8}$, we attempt to find an $n_2$ divisible by 4 such that $I \subseteq J_{n_2,\lambda}$. It would be prudent to consider $f(m) := 1 + \frac{3}{4}m\mu + \lambda\sigma\sqrt{\frac{3}{4}m}$, which is the left endpoint of $J_{m,\lambda}$. We first find an integer $n_2$ such that $f(n_2) \leq a$ and $a - f(n_2)$ is small. 
   By Lemma \ref{lem:A4} in the Appendix, assuming $n$ satisfies bounds (\textbf{B}), it follows that there exists $n_2$ such that:
   \begin{itemize}
   \item $4 \mid n_2$,
   \item $\frac{b^{n-1}}{\mu} \leq n_2 \leq \frac{4}{3\mu}b^{n}$,
   \item $0 \leq a - f(n_2) \leq 3\mu + 1$.
   \end{itemize}
   
    We now check that $I \subseteq J_{n_2,\lambda}$ in order to invoke Theorem \ref{thm:thm1}. 
    We already have that the left endpoint $f(n_2) \leq a$. It remains to check the right endpoints of $I$ and $J_{n_2}$. We need to show that
    \begin{equation} \label{eq:eq12}
    a-1 + \bto{3n}{4} \leq \frac{n_2}{4} + \frac{3n_2}{4}\mu - \lambda\sigma\sqrt{\fk{n_2}}.
    \end{equation}
    The above is equivalent to
    \[ a-\left(\fk{n_2}\mu + \lambda\sigma\sqrt{\fk{n_2}} + 1\right) + \bto{3n}{4} \leq \frac{n_2}{4} - 2 \lambda \sigma \sqrt{\fk{n_2}}. \]
    Simplifying, the above follows from showing that
    \[a - f(n_2) + \bto{3n}{4} + 2 \lambda \sigma \sqrt{\fk{n_2}} \leq \frac{n_2}{4}. \]
    Now let
    \[ \text{LHS} := a - f(n_2) + b^{3n/4} + 2\lambda \sigma \sqrt{\fk{n_2}}.\]
    Then
    \[ \text{LHS} \leq 3\mu + 1 + b^{3n/4} + \lambda\sigma \sqrt{3n_2}. \]
    Using the facts that $\lambda = b^{n/8}$ and $n_2 \leq \frac{4b^{n}}{3\mu}$, we get that
    \[ \text{LHS} \leq 3\mu + 1 + b^{3n/4} + 2\frac{\sigma}{\sqrt{\mu}} b^{5n/8}. \]
    
    Now consider $\text{RHS} := \frac{n_2}{4}$. By the assumptions on $n_2$, we have
    \[\text{RHS} \geq \frac{b^{n}}{4b\mu}. \]
    So (\ref{eq:eq12}) follows from showing that
    \[ 3\mu + 1 + b^{3n/4} + 2\frac{\sigma}{\sqrt{\mu}} b^{5n/8} \leq \frac{b^{n}}{4b\mu}. \]    
    The above is exactly the bound (\textbf{B3}). Therefore, $I \subseteq J_{n_2,\lambda}$. Thus, by applying Theorem \ref{thm:thm1} with $\lambda = \bto{n}{8}$, there exists an $n_2$-strict interval $I_2$  such that 
    
     \[   d_C(I_2) \geq d_C(I) \Big(1-\frac{1}{\bto{n}{4}}\Big)\Big(\frac{1}{1 + \frac{ \sqrt{3n_2} \sigma       \bto{n}{8}}{\bto{3n}{4}}}\Big). \]
    
    Since $n_2 \leq \frac{4}{3\mu}b^n$, it follows that
    \[\frac{1}{1 + \frac{ \sqrt{3n_2} \sigma \bto{n}{8}}{\bto{3n}{4}}} \geq \frac{1}{1+\frac{2\sigma}{\sqrt{\mu}}\bto{-n}{8}} \geq 1 - \frac{2\sigma}{\sqrt{\mu}}\bto{-n}{8}.\]
    Thus, we conclude that $d_C(I_2) \geq d_C(I) \left(1-\bto{-n}{4}\right)\left(1- \frac{2\sigma}{\sqrt{\mu}}\bto{-n}{8}\right)$.
    \end{proof}
    Apply the previous theorem to our starting $n_1$-strict interval $I_1$ to get an $n_2$-strict interval $I_2$. Since $n_2 > n_1$, we can apply Theorem \ref{thm:thm2} again on $I_2$. Continuing in this manner produces a sequence of integers $\{n_i\}_{i=1}^{\infty}$ and $n_i$-strict intervals $\{I_i\}_{i=1}^{\infty}$ such that, for all $i$:
    \begin{itemize}
      \item $n_{i+1} \geq \frac{b^{n_i-1}}{\mu}$,
      \item $d_C(I_{i+1}) \geq d_C(I_i)\left(1-\bto{-n_i}{4}\right)\left(1- \frac{2\sigma}{\sqrt{\mu}}\bto{-n_i}{8}\right).$
    \end{itemize}
    
    The second condition implies that, for all $i$,
    \begin{equation} \label{eq:eq6}
    d_C(I_i) \geq d_C(I_1)\prod\limits_{i=1}^{\infty}\left(\left(1-\bto{-n_i}{4}\right)\left(1- \frac{2\sigma}{\sqrt{\mu}}\bto{-n_i}{8}\right)\right).
    \end{equation}
    
    The following fact will help simplify the above expression. For positive real numbers $x$ and $\alpha$, if $x \geq 2\alpha > 0$, then
    \[ 1 - \alpha x^{-1} \geq \frac{1}{1+2\alpha x^{-1}} \geq e^{-2\alpha x^{-1}}. \]
    Therefore, (\ref{eq:eq6}) implies that
    
    \[ d_C(I_i) \geq d_C(I_1) \cdot \text{exp}\left(\sum\limits_{i=1}^\infty -2\bto{-n_i}{4} - \frac{4\sigma}{\sqrt{\mu}}\bto{-n_i}{8}\right). \]
    For all $i \in \mathbb{N}$, it holds that $n_i \geq in_1$ (it may happen that $n_2 < 2n_1$ if $\mu$ is very large,  but assuming the bounds (\textbf{B}) this will not be the case). The sum in the previous inequality is the sum of two geometric series, one with ratio $r = \bto{-n_1}{4}$ and first term $a = -2\bto{-n_1}{4}$. The second has $r =  \bto{-n_1}{8}$ and $a = \frac{-4\sigma}{\sqrt{\mu}}\bto{-n_1}{8}$. Recall that an infinite geometric series with $|r| < 1$, and first term $a$ sums to
    \[ \frac{a}{1-r}. \]
     Therefore, the first series sums to $\frac{-2\bto{-n_1}{4}}{1-\bto{-n_1}{4}}$, the second sums to $\frac{-4\sigma\bto{-n_1}{8}}{\sqrt{\mu}(1-\bto{-n_1}{8})}$. After simplifying we conclude that, for all $i$,
    \[  d_C(I_i) \geq d_C(I_1) \cdot \text{exp}\left({\frac{-2}{\bto{n_1}{4}-1} + \frac{-4\sigma}{\sqrt{\mu}(\bto{n_1}{8}-1)}}\right). \]
    Thus, we have proven the following
    
    \begin{theorem1} \label{thm:thm3}
      Assume there exists $n_1$ satisfying the bounds (\textbf{B}) and an $n_1$-strict interval $I_1$. Then, for all $N \in \mathbb{N}$, there exists $n > N$ and an $n$-strict interval $I$ such that
     \[  d_C(I) \geq d_C(I_1) \cdot \loss{n_1}. \]
    \end{theorem1}
    
    \section{Main Result}
    As in the previous section we continue to assume that $H$ is a $b$-happy function with digit mean $\mu$ and digit variance $\sigma^2$. Also, we assume that we have experimentally found a suitable starting $n_1$-strict interval, $I_1$, with large type-C density for some $C \subset \mathbb{N}$. As in Section 2, for positive integers $m$, let $Y_m$ denote the random variable induced by the interval $[0,b^m-1]$. 
    
    In this section we give a proof of the following
    
      \begin{theorem1} \label{thm:thm4}
     Suppose $I_1$ is $n_1$-strict, where $n_1$ satisfies bounds (\textbf{B}). Let $\bar{d}$ denote the upper density of the set of type-C integers. Then 
      \[\bar{d} \geq d_C(I_1) \cdot \loss{n_1}.\]
    \end{theorem1}
    The digit mean and digit variance for the case $(e,b)=(2,10)$ are $28.5$ and $721.05$ respectively. In this case, if $n>13$, then it satisfies bounds (\textbf{B}). After performing a computer search we find that the density of happy numbers in the interval $[10^{403},10^{404}-1]$ is at least $.185773$; thus, there exists a $404$-strict interval containing at least this density of happy numbers as a subset. Consider 
    \[\delta(n) := \Big(\frac{2}{1-\bto{n}{4}}+\frac{4\sigma}{\sqrt{\mu}(1-\bto{n}{8})}\Big).\] Plugging in the value for $n$, we find that $e^{\delta(404)} > 1-10^{-49}$. Thus, by Theorem \ref{thm:thm4}, the upper density of type-$\{1\}$ integers is at least $.1857729$. For the lower density, the type-$\{1\}$ density of $[10^{2367},10^{2368}-1]$ is at most $.11379$. This implies that there is a $2368$-strict interval with type-$\{4,16,37,58,89,145,42,20\}$ density at least $1 - .11379$ (recall that there are only two cycles for the $(2,10)$-function). We can then apply the main result to conclude that the upper density of type-$\{4,16,37,58,89,145,42,20\}$ integers is at least $1-.1138$. This gives the following
    \begin{cor}
    Let $\underline{d}$ and $\bar{d}$ be the lower and upper density of $(2,10)$-happy numbers respectively. Then $\underline{d} < .1138$ and $\bar{d} >  .18577$.
    \end{cor}
    
     The proof of Theorem \ref{thm:thm4} is somewhat technical despite having a rather intuitive motivation. For the sake of clarity we first give a sketch of how to use Theorems \ref{thm:thm3} and \ref{thm:thm5} in order to prove a lower bound on the upper density of type-C numbers.
    
    Given our starting interval $I_1$, apply Theorem \ref{thm:thm3} to construct an $n$-strict interval $I$, where \[d_C(I) \geq (1-o(1))d_C(I_1).\] Do this with $n$ large enough as to make all the following error estimations arbitrarily small. Pick $m_1$ such that $\mu m_1$ (i.e., the mean of $H(Y_{m_1})$) lands in the interval $I$. Since $I \subseteq [b^{n-1},b^n]$, we have that $m_1 = \Theta(b^n)$. This implies that the standard deviation of $H(Y_{m_1})$ is roughly $\bto{n}{2}$. This will be much less than $|I| = \bto{3n}{4}$ and thus the bulk of the distribution of $H(Y_{m_1})$ will lie in the interval $I$. 
    
    Next, use Lemma \ref{lem:lem2} with a large $\lambda$ to find an integer $k$ for which \[ \mathbb{P}\left(\tau_k\big(H(Y_{m_1})\big) \text{ is type-C}\right) \geq (1-o(1))d_C(I_1).\] 
    Note that this $k$ will be smaller than $|I| = \bto{3n}{4}$ and that the mean of $\tau_k\big(H(Y_{m_1})\big)$ is equal to $\mu m_1 + k$. Clearly, there exists an integer $m_2$ such that
    \[|\mu m_2 - (\mu m_1 + k)| \leq \mu. \]
    Consider the random variable $H(Y_{m_2})$. The means of $H(Y_{m_2})$ and  $\tau_k \left( H(Y_{m_1})\right)$ are almost equal. Since $k$ is much smaller relative to  $m_1$ and $m_2$, the variance of these two distributions will be close. Furthermore, the distributions of $H(Y_{m_2})$ and $\tau_k \big( H(Y_{m_1}) \big)$ are asymptotically locally normal, so we may apply the local limit law to conclude that the distributions are point-wise close near the means.  Thus, \[\mathbb{P}\big( H(Y_{m_2}) \text{ is type-C}\big) \geq (1-o(1)) d_C(I_1).\]
     This implies that the interval $[0, b^{m_2}-1]$ has type-C density at least $d_C(I_1)\left(1-o(1)\right)$. Note that in the above analysis, we may take $n$ (and therefore $m_2$) to be arbitrarily large. This lower bounds the upper density of type-$C$ integers by $d_C(I_1)(1-o(1))$. In fact, the only contribution to the error term is from the application of Theorem \ref{thm:thm3} (the rest of the error tends to 0 as $n$ tends to infinity).
    
    \subsection{Some Lemmas}
    
    We will now begin to prove the main result. We have broken some of the pieces down for 3 lemmas. The proofs primarily consist of calculations and we leave them for after the proof of the main result. Note that Lemma \ref{lem:lem3} (part 1) is the only place where the local limit law is used.
    
    \begin{lemma1} \label{lem:lem31}
    There exists a sufficiently large $N$ such that, if $n > N$ and $I$ is an $n$-strict interval, then there exists $m \in \mathbb{N}$ with the property that
    \[ [\mu m - \sigma \xto{m}{5}{8}, \mu m + \sigma \xto{m}{5}{8}] \subseteq I. \]
    \end{lemma1}
    
    \begin{lemma1} \label{lem:lem32}
    Let $\epsilon > 0$ be given (assume as well that $\epsilon \leq 1$). Let $\lambda := \sqrt{\frac{6}{\epsilon}}$. Then there exists a sufficiently large $N$ such that, if $n, m_1,$ and $I$ all satisfy:
    \begin{itemize}
      \item $n > N$,
      \item $m_1 \in [\frac{b^{n-1}}{\mu}, \frac{b^{n}}{\mu}]$,
      \item I is $n$-strict,
    \end{itemize}
    then the following hold:       
    \begin{enumerate}
     \item $\lambda \leq \xto{m_1}{1}{8}$,
     \item $\left|1 - \frac{1}{1+ \frac{2\lambda \sigma \sqrt{m_1}}{\bto{3n}{4}}}\right| \leq \epfive$.
    \end{enumerate}
    
    \end{lemma1}

    \begin{lemma1} \label{lem:lem3}
    Let $\epsilon > 0$ be given (assume as well that $\epsilon \leq 1$). Let $T := \frac{2\sqrt{6}}{\sqrt{\epsilon}}$, and $\lambda := \frac{\sqrt{6}}{\sqrt{\epsilon}}$. Then there exists a sufficiently large $N$ such that, if $n, m_1, m_2, k$, and $I$ all satisfy:
    \begin{itemize}
      \item $n > N$,
      \item $m_1 \in [\frac{b^{n-1}}{\mu}, \frac{b^{n}}{\mu}], m_2 \in [\frac{b^{n-2}}{\mu}, \frac{b^{n+1}}{\mu}]$,
      \item $|k| \leq \bto{3n}{4}$,
      \item $|\mu m_1 + k - \mu m_2| \leq \mu$,
      \item I is $n$-strict,
    \end{itemize}
    then the following hold:       
    \begin{enumerate}
     \item For $i\in \{1,2\}$,     $\max\limits_{|t| \leq T} \left|1 - \frac{\mathbb{P}(H(Y_{m_i}) = \lfloor \mu m_i + t \sigma \sqrt{m_i} \rfloor)}{\frac{\xto{e}{t^2}{2}}{\sqrt{2 \pi m_i} \sigma }}\right| \leq \frac{\epsilon}{6}$.
     \item $\left|1- \sqrt{\frac{m_1}{m_2}}\right| \leq \frac{\epsilon}{6}$.
     \item For any real numbers $t_1$ and $t_2$, where $t_1 \in [\frac{-T}{2},\frac{T}{2}]$ and \[\mu m_1 + k + t_1 \sigma \sqrt{m_1} = \mu m_2 + t_2 \sigma \sqrt{m_2},\] it holds that $t_2 \in [-T,T]$ and $|1 - \xto{e}{{t_1}^2 - {t_2}^2}{2}| \leq \frac{\epsilon}{6}$. 
    \end{enumerate}
    \end{lemma1}
    
    \subsection{Proof of Theorem 4.1}
    
    \begin{proof}
    
    In order to lower bound the upper density of type-$C$ integers, it suffices to show that, for all $\epsilon > 0$ and $N_1 \in \mathbb{N}$, there exists $m > N_1$ such that \[d_C([0,b^m-1]) \geq d_C(I_1) \cdot \loss{n}(1 - \epsilon).\] Let $\epsilon$ and $N_1$ be arbitrary (with $\epsilon \leq 1$). Set $T := \frac{2\sqrt{6}}{\sqrt{\epsilon}}$. Also, in anticipation of applying Lemma \ref{lem:lem2}, set $\lambda := \frac{\sqrt{6}}{\sqrt{\epsilon}}$.
    
    First, pick $N >N_1$ large enough to apply Lemmas \ref{lem:lem31}, \ref{lem:lem32}, and \ref{lem:lem3}. By Theorem \ref{thm:thm3}, there exists an $n$-strict interval $I$, where $n > N$ and
    \begin{equation} \label{eq:eq7}
       d_C(I) \geq d_C(I_1) \cdot \loss{n_1}.
    \end{equation}
    
     For $m \in \mathbb{N}$, let \[J_m :=[\mu m - \sigma \xto{m}{5}{8}, \mu m + \sigma \xto{m}{5}{8}].\]
     Recall that $\textbf{E}[H(Y_m)] = \mu m$ and  $\textbf{Var}[H(Y_m)] = \sigma^2 m$.   
    Hence, $J_m$ is where the bulk of the distribution of $H(Y_m)$ lands. Pick $m_1$ such that $J_{m_1} \subseteq I$ (the existence of such $m_1$ is guaranteed by Lemma \ref{lem:lem31}). Note that $m_1 \in [\frac{b^{n-1}}{\mu},\frac{b^{n}}{\mu}]$ since $I$ is $n$-strict. Let $S$ be the set of type-$C$ integers in $I$.  Apply Lemma \ref{lem:lem2} on the random variable $H(Y_{m_1})$ to find an integer $k$ such that
    
    \begin{equation} \label{eq:eq8}
    \mathbb{P}\Big(\tau_k \big (H(Y_{m_1})\big) \in S \Big) \geq d_C(I) \left(1-\frac{1}{\lambda^2}\right)\Big(\frac{1}{1+\frac{2\lambda \sigma \sqrt{m_1}}{|I|}}\Big).
    \end{equation}
    
    Since $J_{m_1} \subseteq I$ and $|I| = \bto{3n}{4}$, it follows that $k \leq \bto{3n}{4}$.

    Let $\tau_k(J_{m_1}):=[a+k,b+k]$, where $J_{m_1} = [a,b]$. Let $S'$ be the set of type-$C$ integers in interval $\tau_k(J_{m_1})$. Recall the proof of Lemma \ref{lem:lem2}. In particular, we applied Lemma \ref{lem:lem1} after ignoring the tails of the distribution of $H(Y_{m_1})$ outside of $\lambda \sigma \sqrt{m_1}$ from the mean. Since $\lambda \leq \xto{m_1}{1}{8}$ (by Lemma \ref{lem:lem32}, part 1), we may replace (\ref{eq:eq8}) by the stronger conclusion that
    
    \[\sum\limits_{i \in S'}\mathbb{P}\Big(\tau_k \big(H(Y_{m_1})\big) = i \Big) \geq d_C(I) \Big(1-\frac{1}{\lambda^2}\Big)\Big(\frac{1}{1+\frac{2\lambda \sigma \sqrt{m_1}}{|I|}}\Big).\]
    
    Using the assumption that $\lambda = \sqrt{\frac{6}{\epsilon}}$ and  part 2 of Lemma \ref{lem:lem32}, we simplify the above as
     \begin{equation} \label{eq:eq9}
       \sum\limits_{i \in S'}\mathbb{P}\Big(\tau_k \big(H(Y_{m_1})\big) = i \Big) \geq d_C(I) \left(1-\epfive\right)^2.
    \end{equation}

   Now pick $m_2 \in \mathbb{N}$ such that $|m_1 \mu + k - m_2 \mu| \leq \mu$. Since $|k| \leq \bto{3n}{4}$, it follows that $m_2 \in [\frac{b^{n-2}}{\mu},\frac{b^{n+1}}{\mu}]$. In particular $m_1,m_2,n,k,$ and $I$ now all satisfy the conditions of Lemma \ref{lem:lem3}. It remains to show that near the mean of $\tau_k\big(H(Y_{m_1})\big)$, the distributions of $\tau_k\big(H(Y_{m_1})\big)$ and $H(Y_{m_2})$ are similar. This will imply that the interval $[0,b^{m_2}-1]$ contains a large density of type-$C$ integers. Making this precise, we prove the following
   
   \begin{claim}
   For integers $i \in \tau_k\big(J_{m_1}\big)$,
      \[\frac{\prtwo}{\prone} \geq \left(1-\epfive\right)^4.\]
   \end{claim}
   \begin{proof}
   Let $i \in \tau_{k}(J_{m_1})$ be fixed and pick $t_1, t_2$ such that 
   \[ i = \mu m_1 + k + t_1 \sigma \sqrt{m_1} = \mu m_2 + t_2 \sigma \sqrt{m_2}. \]
   It is important now that we had chosen $\lambda = \frac{T}{2}$, this implies that $|t_2| \leq T$ (see Lemma \ref{lem:lem3} part 3). We can use the local limit law to estimate the distributions of $\tau_k\big(H(Y_{m_1})\big)$ and $H(Y_{m_2})$. By Lemma \ref{lem:lem3} part 1,
   \[\prtwo = \mathbb{P}\big(H(Y_{m_2}) = \mu m_2 + t_2 \sigma \sqrt{m_2}\big) \geq \frac{\xto{e}{{-t_2}^2}{2}}{2\pi \sigma \sqrt{m_2}}\left(1-\epfive\right)\]
    and
     \[\prone = \mathbb{P}\big(H(Y_{m_1}) = \mu m_1 + t_1 \sigma \sqrt{m_1}\big) \leq \frac{\xto{e}{{-t_1}^2}{2}}{2 \pi \sigma \sqrt{m_1}}\left(1+\epfive\right).\]
      Hence,
   
   \[ \frac{\prtwo}{\prone} \geq \text{exp}\left(({t_1}^2-{t_2}^2)/2\right) \frac{\sqrt{m_1}}{\sqrt{m_2}} \frac{(1-\epfive)}{(1+\epfive)}. \]
   The above, by Lemma \ref{lem:lem3} parts 2 and 3, is at least $\left(1-\epfive\right)^4.$
   \end{proof}
   Putting it all together, we have shown that
   \begin{align*} d_C([0,b^{m_2}-1]) & \geq \sum\limits_{i \in S'}\prtwo  \\
    & =  \sum\limits_{i \in S'}\frac{\prtwo}{\prone} \prone. \\   
   \tag{Claim 1} & \geq \left(1-\epfive\right)^4\sum\limits_{i \in S'} \prone  \\
   \tag{Equation \ref{eq:eq9}} & \geq d_C(I) \left(1-\epfive\right)^6 \\
   & \geq d_C(I)(1-\epsilon).
   \end{align*}  
   To conclude the proof, equation (\ref{eq:eq7}) implies that
   \[ d_C([0,b^{m_2}-1]) \geq d_C(I_1) \cdot \loss{n_1} (1-\epsilon).\] 
  
   \end{proof}
We conclude this section with the proofs of the lemmas used in the previous theorem.
    
\subsection*{Proof of Lemma \ref{lem:lem31}}
   
   \begin{proof}
     For $m \in \mathbb{N}$, let $J_m := [\mu m - \sigma \xto{m}{5}{8},\mu m + \sigma \xto{m}{5}{8}]$. If $I$ is an $n$-strict interval, then $I \subseteq [b^{n-1},b^n-1]$. Note that $\mu m \in I$ implies that $m = O(b^n)$. This in turn shows that \[|J_m| = O(\bto{5n}{8}) \ll |I| = \bto{3n}{4}.\] Comparing the growth rates of $|J_m|$ and $|I|$ it is clear that we can pick $N_1$ large enough such that $n > N_1$ implies that there exists $m$ with $J_m \subseteq I$.
    \end{proof}
    \subsection*{Proof of Lemma \ref{lem:lem32}}
    \begin{proof}
    We find $N_1, N_2$ for the two parts respectively and then choose $N = \max(N_1,N_2)$.
    
 \textbf{1.} $\lambda$ is a fixed constant here and it is assumed that $m_1 \geq \frac{b^{n-1}}{\mu}$, so the result is trivial (this gives $N_1$).

 \textbf{2.} For $x > 0$, to show that $\left|1-\frac{1}{1+x}\right| \leq \epfive$, it is equivalent to prove that
    \[ \left(1-\epfive\right)(1+x) \leq 1 \leq \left(1+ \epfive\right)(1+x). \]
    The above follows if $ x \leq \epfive$. Thus, the result will follow by finding $N$ large enough such that $\frac{2 \sigma \lambda \sqrt{m_1}}{\bto{3n}{4}} \leq \epfive$. Using the assumption that $m_1 \leq \frac{b^{n}}{\mu}$, we get
    \[ \frac{2 \sigma \lambda \sqrt{m_1}}{\bto{3n}{4}} \leq \frac{2 \sigma \lambda}{\sqrt{\mu} \bto{n}{4}}. \]
    This is equivalent to
    \[\frac{12 \sigma \lambda}{\sqrt{\mu} \epsilon} \leq \bto{n}{4}. \]
    Hence, picking $N_2 \geq 4 \log_b(\frac{12 \sigma \lambda}{\sqrt{\mu} \epsilon})$ suffices.
    
    \end{proof}
    
    \subsection*{Proof of Lemma \ref{lem:lem3}}
    \begin{proof}
    We first find $N_1, N_2$, and $N_3$ for the 3 parts respectively, and then define $N := \max(N_1, N_2, N_3)$.

   \textbf{1.} For each $m$, we have
    \[H(Y_{m}) = \sum\limits_{i=1}^{m} H(X_{i}),\]
    where each $X_i$ is uniform in the set $\{0,1,\cdots,b-1\}$. Recall that it is assumed that $H(0) = 0$ and $H(1) = 1$. In particular, the random variables $H(X_i)$ satisfy the aperiodic condition required by Theorem \ref{thm:thm5}. Thus, the result follows from applying Theorem \ref{thm:thm5} to the sum $\sum\limits_{i=1}^{m} H(X_{i})$ with finite interval $[-T,T]$. Fix $M$ large enough such that $m > M$ implies that the $O(\xto{m}{-1}{2})$ term in Theorem \ref{thm:thm5} is less than $\frac{\epsilon}{6}$. By assumption, we have that both $m_1$ and $m_2$ are larger than $\frac{b^{n-2}}{\mu}$. Hence, setting $N_1 = \log_b\left(\mu M\right) +2$ suffices.

    \textbf{2.} Ignoring the square root, it suffices to show that
    \begin{equation} \label {eq:eq10}
     \left|1 - \frac{m_1}{m_2}\right| \leq \epfive.
    \end{equation}
    By assumption
    \[\left|\mu m_1 + k - \mu m_2\right| \leq \mu.\]
    Dividing through by $\mu m_2$, it follows that
    \[\left|1 - \frac{m_1}{m_2} - \frac{k}{\mu m_2}\right| \leq \frac{1}{m_2}. \]
    This implies that
    \[ \frac{k}{\mu m_2} - \frac{1}{m_2} \leq 1 - \frac{m_1}{m_2} \leq \frac{1}{m_2} + \frac{k}{\mu m_2}. \]
    Thus, (\ref{eq:eq10}) follows from showing that $\frac{1}{m_2} + \frac{k}{\mu m_2} \leq \epfive$. Using the assumption that $m_2 \geq \frac{b^{n-2}}{\mu}$ and $|k| \leq \bto{3n}{4}$, it follows that
    \[ \frac{1}{m_2} + \frac{k}{\mu m_2} \leq \mu b^{-(n-2)} + b^{2-(n/4)}. \]
    Therefore, picking $N_2 := \max(\log_b(\frac{12\mu}{\epsilon})+2,4\log_b(\frac{12}{\epsilon})+2)$ suffices.

    \textbf{3.} We first find $N'$ such that $n > N'$ implies that $t_2 \in [-T,T]$. We start with the assumption that
    \[ \mu m_1 + k + t_1\sigma \sqrt{m_1} = \mu m_2 + t_2\sigma \sqrt{m_2}. \]
    Using the facts that $|\mu m_1 + k - \mu m_2| \leq \mu$ and $|t_1| \leq \frac{T}{2}$, this implies that
    \[|t_2| \leq \frac{\mu}{\sigma\sqrt{m_2}} + \frac{T \sqrt{m_1}}{2\sqrt{m_2}}. \]
    We assumed that $m_2 \geq \frac{b^{n-2}}{{\mu}}$. Also, in part (2) we showed that there exists $N_2$ such that $n > N_2$ implies that $\frac{\sqrt{m_1}}{\sqrt{m_2}} \leq \left(1 + \frac{\epsilon}{6}\right) \leq \frac{7}{6}$. Hence, if we take $N' > N_2$, it follows that
    
    \[|t_2| \leq \frac{\mu^2}{\sigma \bto{n-2}{2}} + \frac{7T}{12}. \]
    Pick $N' > N_2$ large enough such that $n > N'$ implies that $\frac{\mu^2}{\sigma \bto{n-2}{2}} \leq \frac{5T}{12}$. This will take care of the size of $t_2$.
    
     Now we must show that there exists $N''$ large enough such that $n > N''$ implies that
    
    \[ \left|1 - \xto{e}{{(t_1}^2 - {t_2}^2)}{2}\right| \leq \frac{\epsilon}{6}. \]
    
    For a real number $x$, if we wish to show that $|1- e^x| \leq \epfive$, it is equivalent to prove that
    \[ \ln \left(1 - \epfive\right) \leq x \leq \ln\left(1 + \epfive\right). \]
    Set $\epstar := \min\left(\ln\left(1+\epfive\right), \left|\ln\left(1 - \epfive\right)\right|\right)$. We find $N''$ such that $n > N''$ implies that
    \[\left|\frac{{t_2}^2 - {t_1}^2}{2} \right| \leq \epstar. \]
    It was assumed that 
    \[ \mu m_1 + k + t_1\sigma \sqrt{m_1} = \mu m_2 + t_2\sigma \sqrt{m_2}.  \]
    Equivalently
    \[ \mu m_1 + k - \mu m_2 = t_2 \sigma \sqrt{m_2} - t_1 \sigma \sqrt{m_1}. \]
    Applying the assumption that the left hand side is at most $\mu$ and dividing both sides by $\sigma \sqrt{m_2}$, we get
    \[ \left|t_2 - t_1 \sqrt{\frac{m_1}{m_2}}\right| \leq \frac{\mu}{\sqrt{m_2}\sigma}. \]
    Rearranging, this gives
    \[ \left|t_2 - t_1 + t_1(1 - \sqrt{\frac{m_1}{m_2}})\right| \leq \frac{\mu}{\sqrt{m_2}\sigma}. \]
    This implies that
    \[ |t_2 - t_1| \leq \frac{\mu}{\sqrt{m_2}\sigma} + \left|t_1(1-\sqrt{\frac{m_1}{m_2}})\right|. \]
    We assumed that $m_2 \geq \frac{b^{n-2}}{\mu}$ and $|t_1| \leq T$. By part (2), if we chose $N'' > N_2$, then \[\left|1- \sqrt{\frac{m_1}{m_2}}\right| \leq \mu b^{-(n-2)} + \bto{-(n-2)}{4}.\] Putting this together, it follows that
     \[ \left|\frac{{t_2}^2 - {t_1}^2}{2}\right| = \left|\left(\frac{t_2 + t_1}{2}\right)\left(t_2 - t_1\right)\right| \leq   T\Big( \frac{\xto{\mu}{3}{2} \bto{-(n-2)}{2}}{\sigma} + T(\mu b^{-(n-2)} + \bto{-(n-2)}{4})\Big). \]
     Now, since $T, \mu, \sigma$, and $b$ are all constants, it follows that the right hand side tends to $0$ as $n$ goes to infinity. Therefore, there exists $N''$ such that $n > N''$ implies that the right hand side is at most $\epstar$. Finally, set $N_3 := \max(N',N'')$.
    \end{proof}
    
    \section{Experimental Data}
    The data\footnote{Data generated by fellow graduate student, Patrick Devlin.} presented in this section is the result of short computer searches, so the bounds surely can be improved with more computing time. Floating point approximation with conservative rounding was used.
    \subsection{Finding an Appropriate $n$-strict Interval}

      If $n$ is divisible by 4 and the interval $[b^{n-1},b^n-1]$ has type-$C$ density $d$, then there exists an $n$-strict interval with type-$C$ density at least $d$ which we may apply Theorem \ref{thm:thm4} to. The type-$C$ density of $[b^{n-1},b^n-1]$ for various $n$ can be quickly calculated by first computing the densities of intervals of the form $[0,b^n-1]$; the algorithm which does this was discussed in Section 2. After an appropriate $n$-strict interval is found, we check to see that $n$ satisfies bounds (\textbf{B}), compute the error term, and find the desired bound. Our results show that in almost all cases, the asymptotic density of type-$C$ numbers does not exist.
     
    \subsection{Explanation of Results}
     The following information is given in tables (in the order of column in which they appear):
     
     \begin{enumerate}
     \item The cycle $C$ in which type-$C$ densities are being computed.
     \item The lower bound on the upper density (UD) implied by Theorem \ref{thm:thm4}.
     \item The upper bound on the lower density (LD) implied by Theorem \ref{thm:thm4}.
     \item The $n$ such that the interval $[b^{n-1},b^n -1]$ is used to find the bound (denoted as UD $n$ or LD $n$).
     \item The $ \delta(n) = \Big(\frac{2}{1-\bto{n}{4}}+\frac{4\sigma}{\sqrt{\mu}(1-\bto{n}{8})}\Big)$ part of the error term for Theorem \ref{thm:thm4} (we only present an upper bound on $|\delta(n)|$, the true number is always negative). In all cases the error is small enough not to affect the bounds as we only give precision of about 5 or 6 decimal places.  
\end{enumerate}        
     
    \subsubsection{Cubing the Digits in Base-10} \label{sec:cubing}
      In this case, if $n> 16$, then it satisfies bounds (\textbf{B}). Table 1 shows the results for the cycles when studying the $(3,10)$-happy function. There are 9 possible cycles. Figure 2 graphs the density of type-$\{1\}$ integers less than $10^n$. It is easy to prove, in this case, that $3 \mid n$ if an only if $n \text{ is type-}\{153\}$.
  
 \begin{table}[htpb]
    \caption{Bounds for the cycles appearing for the $(3,10)$-function}
    \begin{tabular}{ | l | l | l | l | l | l | l |}
    \hline
    Cycle & UD & LD & UD $n$ & LD $n$ & UD $\delta(n)$ & LD $\delta(n)$ \\ \hline
    \{1\} & $ > .028219$ & $ < .0049761$ & $10^{864}$ & $10^{132}$ & $ < 10^{-106}$ & $< 10^{-14}$ \\ \hline
    \{55,250,133\} & $ > .06029$ & $ < .0447701$ & $10^{208}$ & $10^{964}$ & $ < 10^{-24}$ & $< 10^{-118}$ 
  \\
    \hline
    
    \{136,244\} & $ > .024909$ & $ < .006398$ & $10^{204}$ & $10^{420}$ & $ < 10^{-23}$ & $< 10^{-51}$
    \\ \hline
    \{153\} & $ =\frac{1}{3}$ & $ = \frac{1}{3}$ & N/A & N/A & N/A & N/A
    \\ \hline
    
    \{160,217,352\} & $ > .050917$ & $ < .03184$ & $10^{160}$ & $10^{456}$ & $ < 10^{-18}$ & $< 10^{-56}$
    \\ \hline
    
    \{370\} & $ > .19905$ & $ < .16065$ & $10^{276}$ & $10^{560}$ & $ < 10^{-32}$ & $< 10^{-68}$
    \\ \hline
    
    \{371\} & $ > .30189$ & $ < .288001$ & $10^{836}$ & $10^{420}$ & $ < 10^{-102}$ & $< 10^{-50}$
    \\ \hline
    
    \{407\} & $ > .04532$ & $ < .0314401$ & $10^{420}$ & $10^{836}$ & $ < 10^{-50}$ & $< 10^{-103}$
    
    \\ \hline
    
    \{919,1459\} & $ > .04425$ & $ < .01843$ & $10^{916}$ & $10^{120}$ & $ < 10^{-112}$ & $< 10^{-13}$
    
    \\ \hline
    \end{tabular}
  \end{table}
  
   \begin{figure}[htb]
  \caption{Density of type-$\{1\}$ integers in the interval $[0,10^{n}-1]$ for the $(3,10)$-function}
  \includegraphics[scale=.6]{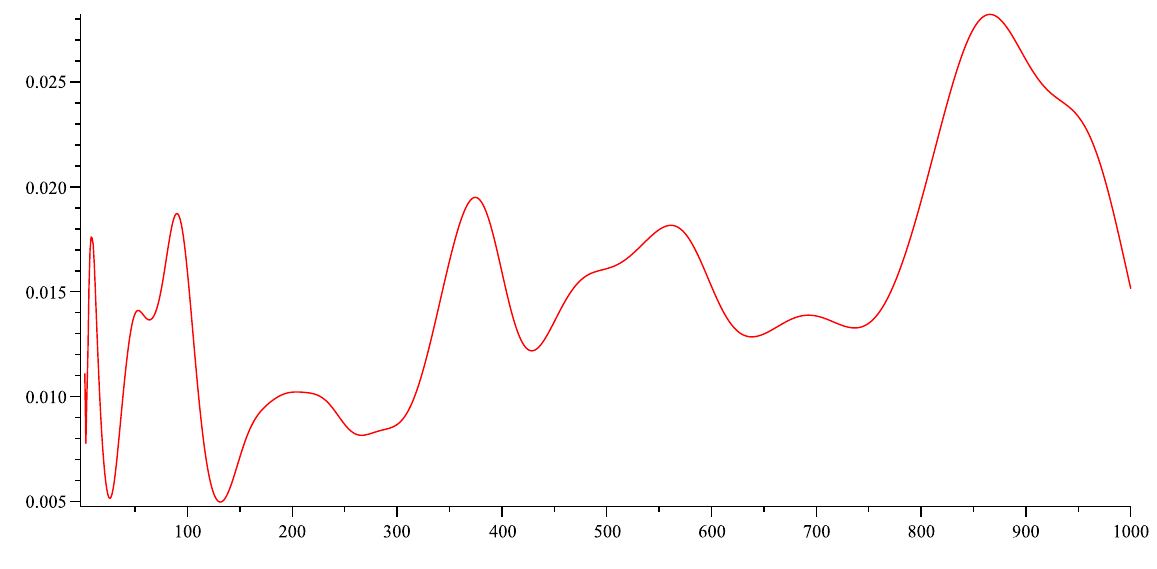} 
 \end{figure}
  
  \subsubsection{A More General Function}
  In order to emphasize the generality of Theorem \ref{thm:thm4}, we consider the function in base-$7$ with digit sequence $[0,1,7,4,17,9,13]$. There are only two cycles for this function, both are fixed points. Written in base-$10$ the cycles are $\{1\}$ and $\{20\}$. Figure 3 graphs the relative density of type-$\{1\}$ numbers. Table 2 shows the bounds derived. As there are only two cycles, we focus on the cycle $\{1\}$. In this case, if $n >12$, then it satisfies bounds (\textbf{B}).
  
 \begin{table}[!htb]
 \caption{Bounds for the cycles appearing for the function with digit sequence $[0, 1,7,4,17,9,13]$}
    \begin{tabular}{ | l | l | l | l | l | l | l |}
    \hline
    Cycle & UD & LD & UD $n$ & LD $n$ & UD $\delta(n)$ & LD $\delta(n)$ \\ \hline
    \{1\} & $ > .9858$ & $ < .94222$ & $7^{176}$ & $7^{384}$ & $ < 10^{-17}$ & $< 10^{-40}$ \\ \hline
    \end{tabular}
 \end{table}
 
 \begin{figure}[!htb]
  \caption{Density of type-$\{1\}$ integers in the interval $[0,7^n-1]$ for Digit Sequence $[0,1,7,4,17,9,13]$}
  \includegraphics[scale=.7]{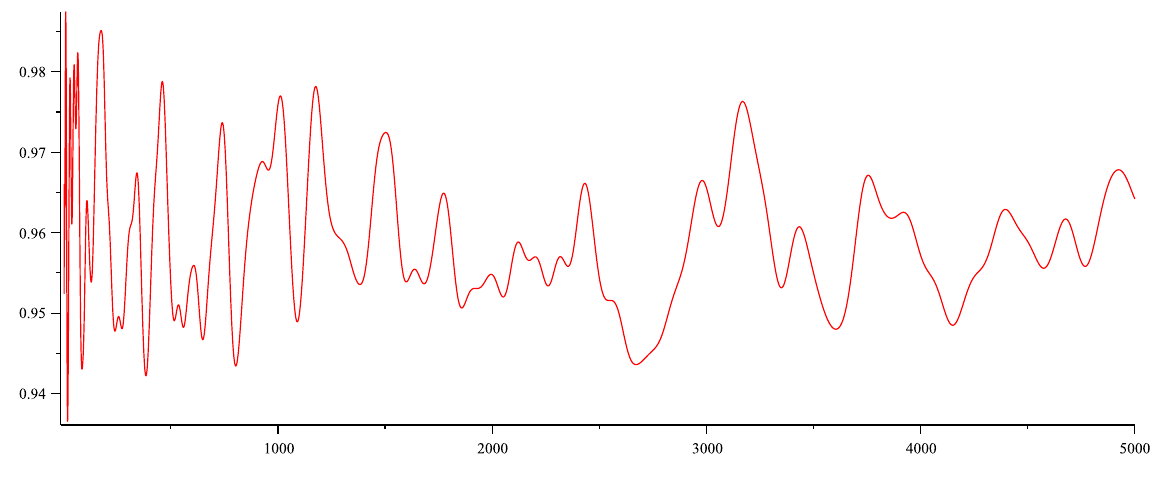} 
 \end{figure}
    
    \section{Appendix}
    
    \begin{lemma1} \label{lem:A1}
   Fix $a > 0$. Assume that $f:\mathbb{R}^{+} \rightarrow \mathbb{R}^{+}$ has continuous first and second derivatives such that, for all $x \in \mathbb{R}^{+}, f'(x) > 0$ and $f''(x) < 0$. Also, assume that $\lim\limits_{x\rightarrow \infty}f(x) = \infty$. Furthermore, suppose we have $x^{*} \in \mathbb{R}^{+}$ such that $f(x^{*} + 1) \leq a$. Then there exists $n \in \mathbb{N}$ such that $n \geq x^*$ and $0 \leq a - f(n) \leq f'(x^{*})$.
   \end{lemma1}
   \begin{proof}
    This follows from a first order Taylor approximation of the function $f$. Let $x^{*}$ such that $f(x^{*} + 1) \leq a$ be given. Set $n := \sup \{m \in \mathbb{N} | f(m) \leq a\}$. Since $f$ is strictly increasing and unbounded this $n$ exists. Note that $f(n) \leq a$ and $f(n+1) > a$. It also follows that $n \geq x^{*}$ as otherwise $\lceil x^{*} \rceil$ would be the supremum. By the concavity of $f$, we have 
    \[f(n+1) - f(n) \leq f'(n) \leq f'(x^*).\] However, $f(n+1) > a$, so we conclude that $0 \leq a - f(n) \leq f'(x^{*})$.
    \end{proof}

    \begin{lemma1}    \label{lem:A4}
    Let $n$ be a positive integer, $\lambda = \bto{n}{8}$, and $a \in [b^{n-1},b^{n}]$. Let $\mu$ and $\sigma$ be the digit mean and variance of some $b$-happy function $H$. Also, assume that $n$ satisfies bounds (\textbf{B}). Let $f(n) := 1 + \frac{3}{4}\mu n + \lambda \sigma \sqrt{\frac{3}{4}n}$. Then there exists an integer $n_2$ such that:
    \begin{itemize}
      \item $\frac{b^{n-1}}{\mu} \leq n_2 \leq \frac{4}{3\mu}b^n$,
      \item $4 \mid n_2$,
      \item $0 \leq a - f(n_2) \leq 3\mu + 1$.
    \end{itemize}
       
    \end{lemma1}
    
    \begin{proof}
      Since we require that $ 4 \mid n_2$, we apply Lemma \ref{lem:A1} on the function  \[g(m) = f(4m) = 1 + 3\mu m + \lambda \sigma \sqrt{3m}.\] Let $x^* := \frac{b^{n-1}}{4\mu}$. We first check that $g(x^* + 1) \leq a$. By assumption, $a \geq b^{n-1}$. Therefore, we need to show that
      \[ 1 + 3\mu\Big( \frac{b^{n-1}}{4\mu} + 1\Big) + \bto{n}{8} \sigma \sqrt{3\Big( \frac{b^{n-1}}{4\mu} + 1\Big)} \leq b^{n-1}. \]
      Simplifying the above, it suffices to show that
      \begin{equation} \label{eq:eqfinal}      
      1 + 3\mu + \bto{5n}{8} \sigma \sqrt{\frac{3}{4b\mu} + 3 b^{-n}} \leq \frac{b^{n-1}}{4}. 
      \end{equation}
      To keep the results of this paper as general as possible, we only assumed that $\mu \geq \frac{1}{b}$ (this would correspond to the quite uninteresting b-happy function $H$ which maps all digits to $0$ except for the digit $1$). Also it is clear that $b^{n} \geq 3$, and therefore 
      \[\frac{3}{4b\mu} + 3 b^{-n} \leq \frac{3}{4} + 1 \leq 2.\]
      
      Plugging this in and rearranging, we see that equation (\ref{eq:eqfinal}) follows if
      \[ 4(1 + 3\mu + \sqrt{2} \sigma \bto{5n}{8}) \leq b^{n-1}.\]
      This is exactly the bound (\textbf{B1}) and is true by assumption. Therefore, by Lemma \ref{lem:A1}, there exists $m \in \mathbb{N}$ such that 
      \[0 \leq a - g(m) \leq g'(x^*).\]
      Also,
      \begin{align*}
      g'(x^*) & =  3\mu + \frac{\sqrt{3}\sigma \bto{n}{8}}{2\sqrt{\frac{b^{n-1}}{4\mu}}} \\
      & = 3\mu + \sqrt{3\mu b}\sigma \bto{-3n}{8}. 
      \end{align*}
      Again, by the assumption (\textbf{B2}) on $n$, the previous statement is bounded above by $3\mu +1$.
      Set $n_2 := 4m$. Then $ 4 \mid n_2, n_2 \geq \frac{b^{n-1}}{\mu}$, and $0 \leq a - f(n_2) \leq 3\mu +1$. Finally, we note that $f(\frac{4}{3\mu}b^n) > a$ and, since $f$ is strictly increasing, we conclude that $n_2 \leq \frac{4}{3\mu}b^n$.
    \end{proof}

    \section{Acknowledgements}
    The author would like to thank Dr. Zeilberger, Dr. Saks, and Dr. Kopparty for their advice and support. He would also like to thank his fellow graduate students Simao Herdade and Kellen Myers for their help in editing, and fellow graduate student Patrick Devlin for generating the necessary data.

\end{document}